\documentclass[11pt,reqno]{amsart}
\newtheorem{theorem}{Theorem}[section]
\newtheorem{proposition}[theorem]{Proposition}
\newtheorem{corollary}[theorem]{Corollary}
\newtheorem{remark}[theorem]{Remark}
\newtheorem{lemma}[theorem]{Lemma}
\newtheorem{example}[theorem]{Example}
\newtheorem{definition}[theorem]{Definition}

\usepackage{amssymb}
\usepackage{graphicx}
\usepackage{amsmath}
\usepackage{enumerate}
\usepackage{color}
\usepackage[all]{xy}
\usepackage[latin1]{inputenc} 
\usepackage{graphicx}
\usepackage{lineno} 
\numberwithin{equation}{section}

\begin{document}\title[ Sania Asif\textsuperscript{1}, Lipeng Luo\textsuperscript{2}, Yanyong Hong\textsuperscript{3} and Zhixiang Wu\textsuperscript{4}
	]{Conformal triple derivations and triple homomorphisms of Lie conformal algebras}
\author{Sania Asif\textsuperscript{1}, Lipeng Luo\textsuperscript{2}, Yanyong Hong\textsuperscript{3} and Zhixiang Wu\textsuperscript{4}}

\address{\textsuperscript{1}Department of Mathematics, Naning University of Information Science and Technology, Nanjing, Jiangsu Province, 210044, PR China.}
\address{\textsuperscript{2}Department of Mathematics, Tongji University, Shanghai, 200092, PR China.}
\address{\textsuperscript{3}Department of Mathematics, Hangzhou Normal University, Hangzhou, 311121, PR China.}
\address{\textsuperscript{4}Department of Mathematics, Zhejiang University, Hangzhou, Zhejiang Province, 310027, PR China.}

\email{\textsuperscript{1}11835037@zju.edu.cn}
\email{\textsuperscript{2}luolipeng1101@163.com}
\email{\textsuperscript{3}hongyanyong2008@yahoo.com}
\email{\textsuperscript{4}wzx@zju.edu.cn}

\keywords{triple derivation, triple homomorphism, conformal algebra}
\subjclass[2000]{Primary 11R52, 15A99, 17B67,17B10, Secondary 16G30}


\date{\textbf{Submitted:} 05 Feburary 2021\\ \textbf{Accepted:} 19 November 2021}
\begin{abstract}
	Let $\mathcal{R}$ be a finite Lie conformal algebra. In this paper, we first investigate the conformal derivation algebra $CDer(\mathcal{R})$, the conformal triple derivation algebra $CTDer(\mathcal{R})$ and the generalized conformal triple derivation algebra $GCTDer(\mathcal{R})$. Mainly, we focus on the connections among these derivation algebras. Next, we give a complete classification of (generalized) conformal triple derivation algebras on all finite simple Lie conformal algebras. In particular, $CTDer(\mathcal{R})= CDer(\mathcal{R})$, where $\mathcal{R}$ is a finite simple Lie conformal algebra. But for $GCDer(\mathcal{R})$, we obtain a conclusion that is closely related to $CDer(\mathcal{R})$. Finally, we introduce the definition of triple homomorphism of a Lie conformal algebra. Furthermore, triple homomorphisms of all finite simple Lie conformal algebras are also characterized.
\end{abstract}

\footnote{The second author is the corresponding author.}
\maketitle

\section{Introduction}\label{intro}
In the last few years, many significant researches have been done in (generalized) derivations of both Lie algebras and their generalizations. It is well known that the research on (generalized) derivations of Lie algebras made a contribution on the development of structure theories of Lie algebras. In particular, many authors have investigated the generalized derivations of non-associative algebras, Lie algebras, Lie superalgebras, Lie color algebras, Lie triple system, Hom Lie triple system and n-Hom Nambu Lie algebras in \cite{Chen-Ma-Ni, HoN, LL, ZZ, Zhou-Chen, Zhou-Chen-Ma-2, Zhou-Chen-Ma-3}. Leger and Lucks have investigated the generalized derivations of Lie algebras and their subalgebras. They have proposed many appealing properties of the generalized derivation algebras and their subalgebras in their work. There are generalizations of derivation such as biderivation and triple derivation. Many authors have investigated the biderivations of the simple generalized Witt algebra, Block Lie algebras and the Schr\"{o}dinger-Virasoro Lie algebra in \cite{CZ, Liu-Guo-Zhao, Wang-Yu}.
\par
In \cite{MD}, M\"{u}ller first introduced triple derivation of Lie algebra as a generalization of derivation, where it is also called as prederivation. Triple derivation of Lie algebra is an analogy of triple derivation of associative algebra and Jordan algebra. It is not difficult to check that, for any Lie algebra, every derivation is a Lie triple derivation, but the converse does not always hold. In recent years, many authors have investigated Lie triple derivations on nest algebras, TUHF algebras, triangular algebras, perfect Lie algebras and perfect Lie superalgebras, Lie color algebras, quaternions algebra in \cite{AW, AWM, Ji-Wang, Lu,  Xiao-Wei, ZJD, Zhou-Chen-Ma-1}. They have obtained many rigorous results from these researches. Lie triple derivations of the Lie algebra of strictly upper triangular matrix over a commutative ring were posed by Wang and Li in \cite{Li-Wang}. In \cite{Li-Wang-Guo}, generalized Lie triple derivations for $gl(n,\mathcal{R})$ were characterized by Li, Wang and Guo. Several problems on Lie derivations, Jordan derivations and homomorphisms of associative algebras were proposed by Herstein in \cite{HN}. There are some derivations on associative algebras proposed by many authors, such as associative triple derivations and Jordan triple derivations in \cite{BC, HN, SGA}.
\par In this paper, we will mention conformal derivations, (generalized) conformal triple derivations on a Lie conformal algebra introduced by Kac in \cite{KacF, KacV}, which is a generalization of Lie algebra. It has been shown that the theory of Lie conformal algebras has close connections to infinite-dimensional Lie algebras satisfying the locality property in \cite{KacL}. Virasoro Lie conformal algebra $Vir$ and current Lie conformal algebra $Cur\mathcal{G}$ associated to a Lie algebra $\mathcal{G}$  are two important examples of Lie conformal algebras. It was shown in \cite{DK} that $Vir$ and all current Lie conformal algebras $Cur\mathcal{G}$, where $\mathcal{G}$ is a finite dimensional simple Lie algebra exhaust all finite simple Lie conformal algebras. Conformal derivations, conformal quasiderivations and generalized conformal derivations of all finite simple Lie conformal algebras were determined by Fan, Hong and Su in \cite{Fan-Hong-Su}. In \cite{Zhao-Chen-Yuan, Zhao-Yuan-Chen}, all kinds of these derivations of Lie conformal superalgebras and Hom-Lie conformal algebras were proposed by Zhao, Chen and Yuan. In this paper, we will classify (generalized) conformal triple derivations for all finite simple Lie conformal algebras.
\par Besides, studies on homomorphisms of Lie algebras and conformal algebras are also very interesting. Many authors have studied the relations of homomorphisms, anti-homomorphisms, Lie homomorphisms and Lie triple homomorphisms. In some suppositions, it turns out that Lie triple homomorphism of associative rings, can be decomposed into the direct sum of associative homomorphisms, associative anti-homomorphisms, Lie homomorphisms and Lie triple homomorphisms and some evident mappings (see \cite{BM, JR, MCR}). Triple homomorphisms of perfect Lie algebras and Lie superalgebras were investigated in \cite{ZJH, Zhou-Chen-Ma-1}. In \cite{BM}, Bresar gave a characterization of Lie triple isomorphisms associated to certain associative algebras. In this paper, we introduce the definition of triple homomorphism of a Lie conformal algebra. Furthermore, the triple homomorphisms of all finite simple Lie conformal algebras are discussed.
\par
The rest of the paper is organized as follows. In Section 2, we introduce basic definitions, notations and some well known results about Lie conformal algebras. In Section 3, the definition and notation of (generalized) conformal triple derivations on Lie conformal algebras are introduced at first. Then we investigate the conformal algebra structure and some properties of these derivations. In Section 4, the (generalized) conformal triple derivations of all simple Lie conformal algebras are classified. In Section 5, we first introduce the definition of triple homomorphism of a Lie conformal algebras. Afterwards, the triple homomorphisms of all finite simple Lie conformal algebras are characterized.
\par Throughout this paper, we use notations $\mathbb{C}$,  $\mathbb{Z}$ and $\mathbb{Z^+}$ to represent the set of complex numbers, integers and nonnegative integers, respectively. In addition, all vector spaces and tensor products are over $\mathbb{C}$. In the absence of ambiguity, we abbreviate $\otimes_{\mathbb{C}}$ by $\otimes$.

\section{preliminaries}
In this section, we recall some basic definitions, notations and related results about Lie conformal algebras for later use. For a detailed description, one can refer to \cite{DK, KacV}.
\begin{definition}\label{def2.1}
	\em
		{A \emph {Lie conformal algebra} $\mathcal{R}$ is a $\mathbb{C}[\partial]$-module endowed with a $\mathbb{C}$-linear map from $\mathcal{R} \otimes \mathcal{R}$ to $\mathbb{C}[\lambda] \otimes \mathcal{R}, a \otimes b \mapsto [a_\lambda b]$, called the $\lambda$-bracket, satisfying the following axioms:
		\begin{align}
		[\partial a_\lambda b]&= -\lambda[a_\lambda b], \quad  [ a_\lambda \partial b]= (\partial+ \lambda)[a_\lambda b] \quad (conformal \  sesquilinearity),\\
		{}[a_\lambda b]&= -[b_{-\lambda- \partial} a]\quad (skew\text{-}symmetry),\\
		{}[a_\lambda [b_\mu c]]&= [[a_\lambda b]_{\lambda+ \mu} c]+ [ b_\mu [a_\lambda c]]\quad (Jacobi\  identity),
		\end{align}
		for $a, b, c \in\mathcal{R}$.}
	
\end{definition}

The notions of subalgebras, ideals, quotients and homomorphisms of Lie conformal algebras are obvious. Moreover, due to the skew\text{-}symmetry, any left or right ideal is actually a two-side ideal. A Lie conformal algebra $\mathcal{R}$ is called \emph {finite} if $\mathcal{R}$ is finitely generated as a $\mathbb{C}[\partial]$-module. The \emph {rank} of a Lie conformal algebra $\mathcal{R}$, denoted by rank($\mathcal{R}$), is its rank as a $\mathbb{C}[\partial]$-module. A Lie conformal algebra $\mathcal{R}$ is \emph{simple} if it has no non-trivial ideals and it is not abelian.

Virasoro conformal algebra $Vir$ and current Lie conformal algebra $Cur\mathcal{G}$ associated to a Lie algebra $\mathcal{G}$  are two important examples of Lie conformal algebras.

\begin{example}\label{ex2.2}
	\em{
		The conformal algebra $Vir$ is a free $\mathbb{C}[\partial]$-module on the generator $L$, which is defined as follows
		\begin{align*}
		Vir= \mathbb{C}[\partial]L, \quad [L_\lambda L]= (\partial+ 2\lambda)L.
		\end{align*}}

\end{example}
\begin{example}\label{ex2.3}
	\em{
		Let $\mathcal{G}$ be a Lie algebra. The current Lie conformal algebra associated to $\mathcal{G}$ is defined as follows
		\begin{align*}
		Cur\mathcal{G}= \mathbb{C}[\partial] \otimes \mathcal{G}, \quad [a_\lambda b]= [a, b],
		\end{align*}
		for all $a, b \in \mathcal{G}$.}
	
\end{example}
It was shown in \cite{DK} that $Vir$ and all current Lie conformal algebras $Cur\mathcal{G}$ where $\mathcal{G}$ is a finite dimensional simple Lie algebra exhaust all finite simple Lie conformal algebras. In this paper, we will classify conformal triple derivations and triple homomorphisms of all finite simple Lie conformal algebras.

\begin{definition}\label{def2.4}
	\em{
		A \emph {conformal linear map} between $\mathbb{C}[\partial]$-modules $\mathcal{A}$ and $\mathcal{B}$ is a $\mathbb{C}$-linear map $\phi_{\lambda}$: $\mathcal{A} \to \mathbb{C}[\lambda] \otimes \mathcal{B}$, satisfying the following axiom:
		\begin{align*}
		\phi_{\lambda}(\partial a)= (\partial+ \lambda)\phi_{\lambda}(a), \quad \forall  a\in\mathcal{A}.
		\end{align*}}
			\end{definition}
Obviously, $\phi_{\lambda}$ does not depend on the choice of the indeterminate variable $\lambda$. The vector space of all conformal linear maps from $\mathcal{A}$ to $\mathcal{B}$ is denoted by $Chom(\mathcal{A}, \mathcal{B})$, which can be made into a $\mathbb{C}[\partial]$-module via $(\partial \phi)_{\lambda}(a)= -\lambda\phi_{\lambda}(a), \forall  a\in\mathcal{A}$. For convenience, we write $Cend{\mathcal{A}}$ for $Chom(\mathcal{A}, \mathcal{A})$.

\begin{definition}\label{def2.5}
	\em{Let $\mathcal{R}$ be a Lie conformal algebra. A conformal linear map $\phi_{\lambda}$: $\mathcal{R} \to \mathbb{C}[\lambda]\otimes\mathcal{R}$ is a \emph {conformal derivation} of $\mathcal{R}$ if
		\begin{align*}
		\phi_{\lambda}([a_{\mu}b])= [\phi_{\lambda}(a)_{\lambda+ \mu}b]+ [a_{\mu}\phi_{\lambda}(b)], \quad \forall a, b \in \mathcal{R}.
		\end{align*}}
		\end{definition}

By the Jacobi identity, it is not difficult to check that for every $a \in \mathcal{R}$, the map $ad a _{\lambda}$ which is defined by $ad a _{\lambda} b= [a_{\lambda} b]$ for any $b \in \mathcal{R}$, is a conformal derivation of $\mathcal{R}$. All conformal derivations of this kind are called \emph{inner conformal derivations}. Denote by $CDer(\mathcal{R})$ and $CInn(\mathcal{R})$ the vector spaces of all conformal derivations and inner conformal derivations of $\mathcal{R}$ respectively.

There is an important example of a non-inner conformal derivation defined as follows,
\begin{example}\label{ex2.6}
	\em{Let $Cur\mathcal{G}$ be the current Lie conformal algebra associated to the finite dimensional Lie algebra $\mathcal{G}$. Define a conformal linear map $d^L_{\lambda}$: $Cur\mathcal{G} \to Cur\mathcal{G}$ by $d^L_{\lambda}a= (\partial+ \lambda)a$ for every $a \in \mathcal{G}$. It is not difficult to check that $d^L_{\lambda}$ is a conformal derivation.}

\end{example}

\begin{definition}\label{def2.7}
	\em{Let $\mathcal{R}$ be a finite Lie conformal algebra. The $\lambda$-bracket on $Cend(\mathcal{R})$ is given by
	\begin{align*}
	[\phi_{\lambda}\psi]_{\mu}a= \phi_{\lambda}(\psi_{\mu- \lambda}a)- \psi_{\mu- \lambda}(\phi_{\lambda}a), \quad \forall a \in \mathcal{R},
	\end{align*}
	defines a Lie conformal algebra structure on $Cend(\mathcal{R})$. This is called the \emph{general Lie conformal algebra} on $\mathcal{R}$ which is denoted by $gc(\mathcal{R})$.}

\end{definition}

Inspired by the above definition, it is easy to check that $CDer(\mathcal{R})$ and $CInn(\mathcal{R})$ are Lie conformal subalgebras of $gc(\mathcal{R})$, where $\mathcal{R}$ is a finite Lie conformal algebra. Moreover, we have the following tower $CInn(\mathcal{R})\subseteq CDer(\mathcal{R})\subseteq gc(\mathcal{R})$.

\section{Conformal triple derivations on Lie conformal algebras}

In this section, we introduce the definition and notation of (generalized) conformal triple derivations on Lie conformal algebras at first. Then we investigate the conformal algebra structure and some properties of these derivations. For convenience, we always assume that $\mathcal{R}$ is a finite Lie conformal algebra in this section unless otherwise specified.

\begin{definition}\label{def3.1}
	\em{
		A conformal linear map $\phi_{x} \in gc(\mathcal{R})$ is called a \emph{conformal triple derivation} of $\mathcal{R}$ if it satisfies the following axiom:
		\begin{align}\label{eq3.1}
		\phi_x([[a_{\lambda}b]_{\lambda+ \mu}c])= [[\phi_x(a)_{\lambda+ x}b]_{\lambda+ \mu+x}c]+ [[a_{\lambda}\phi_x(b)]_{\lambda+ \mu+ x}c]+ [[a_{\lambda}b]_{\lambda+ \mu}\phi_x(c)],
		\end{align}
		for any $a,b,c \in \mathcal{R}$.
		
		A conformal linear map $\phi_{x} \in gc(\mathcal{R})$ is said to be a \emph{generalized conformal triple derivation} of $\mathcal{R}$ if there exists a conformal triple derivation $\tau_x$ of $\mathcal{R}$ such that
		\begin{align}\label{eq3.2}
		\phi_x([[a_{\lambda}b]_{\lambda+ \mu}c])= [[\phi_x(a)_{\lambda+ x}b]_{\lambda+ \mu+ x}c]+ [[a_{\lambda}\tau_x(b)]_{\lambda+ \mu+ x}c]+ [[a_{\lambda}b]_{\lambda+ \mu}\tau_x(c)],
		\end{align}
		for any $a, b, c \in \mathcal{R}$, and $\tau_x$ is called the \emph{conformal triple derivation} related to $\phi_x$. }
		
\end{definition}
Denote by $CTDer(\mathcal{R})$ and $GCTDer(\mathcal{R}),$ the sets of conformal triple derivations and generalized conformal triple derivations of $\mathcal{R}$ respectively. Obviously, conformal triple derivations are all generalized conformal triple derivations, i.e. $CTDer(\mathcal{R})\subseteq GCTDer(\mathcal{R})$ by setting $\tau_x=\phi_{x}$ in (\ref{eq3.2}). However, the converse is not true in general.

By the above definition, we can obtain the following result immediately.
\begin{lemma}\label{lm3.1}
	For any $a, b, c \in \mathcal{R}$, let $\phi_x \in GCTDer(\mathcal{R})$ and $\tau_x$ be the conformal triple derivation related to $\phi_x$. Then we obtain that
	\begin{enumerate}
		\item $\phi_x([a_{\lambda}[b_{\mu}c]])= [\phi_x(a)_{\lambda+ x}[b_{\mu}c]]+ [a_{\lambda}[\tau_x(b)_{\mu+ x}c]]+ [a_{\lambda}[b_{\mu}\tau_x(c)]],$
		\item $(\phi_x- \tau_x)([[a_{\lambda}b]_{\lambda+ \mu}c])= [[(\phi_x- \tau_x)(a)_{\lambda+ x}b]_{\lambda+ \mu+ x}c]=[[a_{\lambda}(\phi_x- \tau_x)(b)]_{\lambda+ \mu+ x}c]\\
		=[[a_{\lambda}b]_{\lambda+ \mu}(\phi_x- \tau_x)(c)]$,
		\item $(\phi_x- \tau_x)([a_{\lambda}[b_{\mu}c]])= [(\phi_x- \tau_x)(a)_{\lambda+ x}[b_{\mu}c]]= [a_{\lambda}[(\phi_x- \tau_x)(b)_{\mu+ x}c]]= [a_{\lambda}[b_{\mu}(\phi_x- \tau_x)(c)]].$
	\end{enumerate}
\end{lemma}
\begin{proof}
	(1) By Jacobi identity, we have that
	\begin{align*}
	\phi_x([a_{\lambda}[b_{\mu}c]])= &\phi_x([[a_\lambda b]_{\lambda+ \mu} c])+ \phi_x([ b_\mu [a_\lambda c]])\\
	=& \phi_x([[a_\lambda b]_{\lambda+ \mu} c])- \phi_x([[a_\lambda c]_{- \partial- \mu}b])\\
	=&[[\phi_x(a)_{\lambda+ x}b]_{\lambda+ \mu+ x}c]+ [[a_{\lambda}\tau_x(b)]_{\lambda+ \mu+ x}c]+ [[a_{\lambda}b]_{\lambda+ \mu}\tau_x(c)]\\
	&-[[\phi_x(a)_{\lambda+ x}c]_{- \partial- x- \mu+ x}b]- [[a_{\lambda}\tau_x(c)]_{- \partial- x- \mu+ x}b]- [[a_{\lambda}c]_{- \partial- x- \mu}\tau_x(b)]\\
	=&[[\phi_x(a)_{\lambda+x}b]_{\lambda+\mu+x}c]+[[a_{\lambda}\tau_x(b)]_{\lambda+ \mu+ x}c]+ [[a_{\lambda}b]_{\lambda+ \mu}\tau_x(c)]\\
	&+ [b_\mu[\phi_x(a)_{\lambda+ x}c]]+ [b_\mu[a_{\lambda}\tau_x(c)]]+ [\tau_x(b)_{\mu+ x}[a_{\lambda}c]]\\
	=& [\phi_x(a)_{\lambda+ x}[b_{\mu}c]]+ [a_{\lambda}[\tau_x(b)_{\mu+ x}c]] + [a_{\lambda}[b_{\mu}\tau_x(c)]].
	\end{align*}
	(2) Since $\tau_{x}$ is the conformal triple derivation related to  $\phi_x$, we deduce that
	\begin{align*}
	\phi_x([[a_{\lambda}b]_{\lambda+ \mu}c])&= [[\phi_x(a)_{\lambda+ x}b]_{\lambda+ \mu+ x}c]+ [[a_{\lambda}\tau_x(b)]_{\lambda+ \mu+ x}c]+ [[a_{\lambda}b]_{\lambda+ \mu}\tau_x(c)]\\
	&= [[\phi_x(a)_{\lambda+ x}b]_{\lambda+ \mu+ x}c]+ \tau_x([[a_{\lambda}b]_{\lambda+ \mu}c])- [[\tau_x(a)_{\lambda+ x}b]_{\lambda+ \mu+ x}c].
	\end{align*}
	That is,
	\begin{align}\label{aa}
	(\phi_x- \tau_x)([[a_{\lambda}b]_{\lambda+ \mu}c])= [[(\phi_x- \tau_x)(a)_{\lambda+ x}b]_{\lambda+ \mu+ x}c].
	\end{align}
	
	Similarly, by skew-symmetry and the equation (\ref{aa}), we can obtain that
	\begin{align*}
	(\phi_x- \tau_x)([[a_{\lambda}b]_{\lambda+ \mu}c])&=-(\phi_x- \tau_x)([[b_{-\lambda-\partial}a]_{\lambda+ \mu}c])\\
	&=-(\phi_x- \tau_x)([[b_{-\lambda-(-\lambda-\mu)}a]_{\lambda+\mu}c])\\
	&=-(\phi_x- \tau_x)([[b_{\mu}a]_{\lambda+\mu}c])\\
	&=-[[(\phi_x- \tau_x)(b)_{\mu+x}a]_{\lambda+\mu+x}c]\\	
	&=[[a_{-x-\mu-\partial}(\phi_x- \tau_x)(b)]_{\lambda+\mu+x}c]\\	
	&=[[a_{-x-\mu-(-\lambda-\mu-x)}(\phi_x- \tau_x)(b)]_{\lambda+\mu+x}c]\\	
	&= [[a_{\lambda}(\phi_x- \tau_x)(b)]_{\lambda+ \mu+x}c].
	\end{align*}
	By skew-symmetry and Jacobi identity, we deduce that
	\begin{align*}
	(\phi_x- \tau_x)([[a_{\lambda}b]_{\lambda+ \mu}c])&= (\phi_x- \tau_x)([a_\lambda [b_\mu c]])- (\phi_x- \tau_x)([ b_\mu [a_\lambda c]])\\
	&=- (\phi_x- \tau_x)([[b_\mu c]_{-\partial- \lambda}a])+ (\phi_x-\tau_x)([[a_\lambda c]_{- \partial- \mu}b])\\
	&=- [[b_\mu (\phi_x- \tau_x)(c)]_{- \partial- \lambda}a]+ [[a_\lambda (\phi_x- \tau_x)(c)]_{- \partial- \mu}b]\\
	&= [a_\lambda[b_\mu (\phi_x- \tau_x)(c)]]- [b_\mu[a_\lambda (\phi_x- \tau_x)(c)]]\\
	&= [[a_{\lambda}b]_{\lambda+ \mu}(\phi_x- \tau_x)(c)].
	\end{align*}
	Hence, we have
	\begin{align*}
	(\phi_x- \tau_x)([[a_{\lambda}b]_{\lambda+ \mu}c])&= [[(\phi_x- \tau_x)(a)_{\lambda+ x}b]_{\lambda+ \mu+ x}c]= [[a_{\lambda}(\phi_x- \tau_x)(b)]_{\lambda+ \mu+ x}c]\\
	&= [[a_{\lambda}b]_{\lambda+ \mu}(\phi_x- \tau_x)(c)].
	\end{align*}
	(3) It follows immediately from (2).
	
	This completes the proof.
\end{proof}

\begin{remark}\label{rm3.3}
	\em{
		Inspired by the proof of (2) in Lemma \ref{lm3.1}, we can obtain some interesting conclusions as follows. Let $\phi_x \in gc(\mathcal{R})$. For any $a,b,c \in \mathcal{R}$, we have\\
		(1) The following three equalities are equivalent:
		\begin{enumerate}[(i)]\item
			$\phi_x([[a_{\lambda}b]_{\lambda+ \mu}c])= [[\phi_x(a)_{\lambda+ x}b]_{\lambda+ \mu+ x}c].$
			\item $\phi_x([[a_{\lambda}b]_{\lambda+ \mu}c])= [[a_{\lambda}\phi_x(b)]_{\lambda+ \mu+ x}c].$
			\item $\phi_x([[a_{\lambda}b]_{\lambda+ \mu}c])= [[\phi_x(a)_{\lambda+ x}b]_{\lambda+ \mu+ x}c]= [[a_{\lambda}\phi_x(b)]_{\lambda+ \mu+ x}c]= [[a_{\lambda}b]_{\lambda+ \mu}\phi_x(c)].$ \end{enumerate}
		(2) The following three equalities are also equivalent:
		\begin{enumerate}[(i)]
			\item $[[\phi_x(a)_{\lambda+ x}b]_{\lambda+ \mu+ x}c]= [[a_{\lambda}b]_{\lambda+ \mu}\phi_x(c)].$
			\item $[[a_{\lambda}\phi_x(b)]_{\lambda+ \mu+ x}c]= [[a_{\lambda}b]_{\lambda+ \mu}\phi_x(c)].$
			\item $[[\phi_x(a)_{\lambda+ x}b]_{\lambda+ \mu+ x}c]= [[a_{\lambda}\phi_x(b)]_{\lambda+ \mu+ x}c]= [[a_{\lambda}b]_{\lambda+ \mu}\phi_x(c)].$
		\end{enumerate}}
	
\end{remark}

\begin{definition}\label{def3.3}
	\em{
		A conformal linear map $\phi_{x} \in gc(\mathcal{R})$ is called a \emph{conformal triple centroid} of $\mathcal{R}$ if it satisfies the following axiom:
		\begin{align}\label{eq3.3}
		\phi_x([[a_{\lambda}b]_{\lambda+ \mu}c])= [[\phi_x(a)_{\lambda+ x}b]_{\lambda+ \mu+ x}c],
		\end{align}
		for any $a, b, c \in \mathcal{R}$.
		
		A conformal linear map $\phi_{x} \in gc(\mathcal{R})$ is said to be a \emph{conformal triple quasicentroid} of $\mathcal{R}$ if it satisfies the following equality:
		\begin{align}\label{eq3.4}
		[[\phi_x(a)_{\lambda+ x}b]_{\lambda+ \mu+ x}c]= [[a_{\lambda}b]_{\lambda+ \mu}\phi_x(c)],
		\end{align}
		for any $a, b, c \in \mathcal{R}$.
		
		A conformal linear map $\phi_{x} \in gc(\mathcal{R})$ is said to be a \emph{central conformal triple derivation} of $\mathcal{R}$ if it satisfies the following equality:
		\begin{align}\label{eq3.5}
		\phi_x([[a_{\lambda}b]_{\lambda+ \mu}c])= [[\phi_x(a)_{\lambda+ x}b]_{\lambda+ \mu+ x}c]=0,
		\end{align}
		for any $a, b, c \in \mathcal{R}$.
		}
		
\end{definition}

Hereafter, we denote by $TC(\mathcal{R})$, $TQC(\mathcal{R})$ and $ZTDer(\mathcal{R})$ the sets of conformal triple centroids, quasicentroids and central conformal triple derivations respectively. Obviously, we can obtain the tower
$ZTDer(\mathcal{R})\subseteq TC(\mathcal{R})\subseteq TQC(\mathcal{R})$ follows from (1) in remark \ref{rm3.3}. As usual, the center of Lie conformal algebra $\mathcal{R}$ is denoted by $Z(\mathcal{R})= \{a\in \mathcal{R} | [a_{\lambda}b]= 0, \ \ \forall b\in \mathcal{R} \}$. For a subset $S$ of $\mathcal{R}$, denote by $C_{\mathcal{R}}(S)$ the centralizer of $S$ in $\mathcal{R}$. It is clear that, if $Z(\mathcal{R})= 0$, then $ZTDer(\mathcal{R})= 0$.

\begin{remark}\label{rm3.5}
	\em{
		Inspired by the result of (2) in Lemma \ref{lm3.1}, we deduce that if $\phi_x \in GCTDer(\mathcal{R})$ and $\tau_x$ is the conformal triple derivation related to $\phi_x ,$ then $\phi_x- \tau_x \in TC(\mathcal{R})$.}
		
\end{remark}

It is well known that $CDer(Vir)= CInn(Vir)$ (see \cite{DK}). Then we investigate the characterization of $ZTDer(Vir)$, $TC(Vir)$ and $TQC(Vir)$. $CTDer(Vir)$ and $GCTDer(Vir)$ are shown in Section $4$ (see Theorem \ref{tm4.8}).
\begin{proposition}\label{pro3.33}
	$TQC(Vir)= TC(Vir)= ZTDer(Vir)= 0$.
\end{proposition}
\begin{proof}
	For any $\phi_x \in TQC(Vir)$, assume that $\phi_x(L)= f(\partial, x)L$ for some $f(\partial, x) \in \mathbb{C}[\partial, x]$. By (\ref{eq3.4}), we obtain that
	\begin{align*}
	f(-\lambda- x, x)(\lambda+ x- \mu)(\partial+ 2\lambda+ 2\mu+ 2x)= f(\partial+ \lambda+ \mu, x)(\lambda- \mu)(\partial+ 2\lambda+ 2\mu).
	\end{align*}
	Obviously, we can deduce that $f(\partial, x)= 0$ by comparing the coefficient of $\partial$ in the above equation. Thus, we have $TQC(Vir)= 0$. According to the above discussion, we have the tower $ZTDer(Vir)\subseteq TC(Vir)\subseteq TQC(Vir)$. Hence, we deduce that $TQC(Vir)= TC(Vir)= ZTDer(Vir)= 0$.
\end{proof}

From now on, we investigate the connections of these conformal triple derivations.
\begin{proposition}\label{pro3.4}
	$CTDer(\mathcal{R}), GCTDer(\mathcal{R})$ and $TC(\mathcal{R})$ are Lie conformal subalgebras of $gc(\mathcal{R})$.
\end{proposition}
\begin{proof}
	It is not difficult to see that the proofs of all cases are similar. So we only need to prove that, for any $\phi_x, \psi_x \in CTDer(\mathcal{R})$, $[\phi_x\psi]_y \in CTDer(\mathcal{R})[x]$.
	
	Since $\phi_x, \psi_x \in CTDer(\mathcal{R})$, we can obtain that
	\begin{align*}
	\phi_x([[a_{\lambda}b]_{\lambda+ \mu}c])= [[\phi_x(a)_{\lambda+ x}b]_{\lambda+ \mu+ x}c]+ [[a_{\lambda}\phi_x(b)]_{\lambda+  \mu+ x}c]+ [[a_{\lambda}b]_{\lambda+ \mu}\phi_x(c)],\\
	\psi_x([[a_{\lambda}b]_{\lambda+ \mu}c])= [[\psi_x(a)_{\lambda+ x}b]_{\lambda+ \mu+ x}c]+ [[a_{\lambda}\psi_x(b)]_{\lambda+ \mu+ x}c]+ [[a_{\lambda}b]_{\lambda+ \mu}\psi_x(c)],
	\end{align*}
	for any $a, b, c \in \mathcal{R}$. Furthermore, we can get
	\begin{align*}
	\phi_x&\psi_{y- x}([[a_{\lambda}b]_{\lambda+ \mu}c])\\
	= &\phi_x([[\psi_{y- x}(a)_{\lambda+ y- x}b]_{\lambda+ \mu+ y- x}c]+ [[a_{\lambda}\psi_{y- x}(b)]_{\lambda+ \mu+ y- x}c]+ [[a_{\lambda}b]_{\lambda+ \mu}\psi_{y- x}(c)])\\
	= &[[\phi_x\psi_{y- x}(a)_{\lambda+ y}b]_{\lambda+ \mu+ y}c]+ [[\psi_{y- x}(a)_{\lambda+ y- x}\phi_x(b)]_{\lambda+ \mu+ y}c]+ [[\psi_{y- x}(a)_{\lambda+ y- x}b]_{\lambda+ \mu+ y- x}\phi_x(c)]\\
	&+[[\phi_x(a)_{\lambda+ x}\psi_{y- x}(b)]_{\lambda+ \mu+ y}c]+ [[a_{\lambda}\phi_x\psi_{y- x}(b)]_{\lambda+ \mu+ y}c]+ [[a_{\lambda}\psi_{y- x}(b)]_{\lambda+ \mu+ y- x}\phi_x(c)]\\
	&+[[\phi_x(a)_{\lambda+ x}b]_{\lambda+ \mu+ x}\psi_{y- x}(c)]+ [[a_{\lambda}\phi_x(b)]_{\lambda+ \mu+ x}\psi_{y- x}(c)]+ [[a_{\lambda}b]_{\lambda+ \mu}\phi_x\psi_{y- x}(c)].
	\end{align*}
	Similarly, we have
	\begin{align*}
	\psi_{y- x}&\phi_x([[a_{\lambda}b]_{\lambda+ \mu}c])\\
	= &[[\psi_{y- x}\phi_x(a)_{\lambda+ y}b]_{\lambda+ \mu+ y}c]+ [[\phi_x(a)_{\lambda+ x}\psi_{y- x}(b)]_{\lambda+ \mu+ y}c]+ [[\phi_x(a)_{\lambda+ x}b]_{\lambda+ \mu+ x}\psi_{y- x}(c)]\\
	&+ [[\psi_{y- x}(a)_{\lambda+ y- x}\phi_x(b)]_{\lambda+ \mu+ y}c]+ [[a_{\lambda}\psi_{y- x}\phi_x(b)]_{\lambda+ \mu+ y}c]+ [[a_{\lambda}\phi_x(b)]_{\lambda+ \mu+ x}\psi_{y- x}(c)]\\
	&+ [[\psi_{y- x}(a)_{\lambda+ y- x}b]_{\lambda+ \mu+ y- x}\phi_x(c)]+ [[a_{\lambda}\psi_{y- x}(b)]_{\lambda+ \mu+ y- x}\phi_x(c)]+ [[a_{\lambda}b]_{\lambda+ \mu}\psi_{y- x}\phi_x(c)].
	\end{align*}
	Thus, we can obtain that
	\begin{align*}
	[\phi_x\psi]_y([[a_{\lambda}b]_{\lambda+ \mu}c])= [[[\phi_x\psi]_y(a)_{\lambda+ y}b]_{\lambda+ \mu+ y}c]+ [[a_{\lambda}[\phi_x\psi]_y(b)]_{\lambda+ \mu+ y}c]+ [[a_{\lambda}b]_{\lambda+ \mu}[\phi_x\psi]_y(c)].
	\end{align*}
	Hence, $[\phi_x\psi]_y \in CTDer(\mathcal{R})[x]$, i.e. $CTDer(\mathcal{R})$ is a Lie conformal subalgebra of $gc(\mathcal{R})$.
	
	This completes the proof.
\end{proof}

\begin{proposition}\label{pro3.5}
	$ZTDer(\mathcal{R})$ is a Lie conformal ideal of $CTDer(\mathcal{R})$ and $GCTDer(\mathcal{R})$.
\end{proposition}
\begin{proof}
	Obviously, we have the tower $ZTDer(\mathcal{R})\subseteq CTDer(\mathcal{R})\subseteq GCTDer(\mathcal{R})$. We only need to prove that, for any $\phi_x\in ZTDer(\mathcal{R})$ and $\psi_x\in CTDer(\mathcal{R})$, then $[\phi_x\psi]_y \in ZTDer(\mathcal{R})[x]$. The other case can be proved similarly.
	
	For any $a, b, c \in \mathcal{R}$, we have
	\begin{align*}
	[[[\phi_x\psi]_y(a)_{\lambda+ y}b]_{\lambda+ \mu+ y}c]= &[[\phi_x\psi_{y- x}(a)_{\lambda+ y}b]_{\lambda+ \mu+ y}c]- [[\psi_{y- x}\phi_x(a)_{\lambda+ y}b]_{\lambda+ \mu+ y}c]\\
	= &0-(\psi_{y- x}([[\phi_x(a)_{\lambda+ x}b]_{\lambda+ \mu+ x}c])- [[\phi_x(a)_{\lambda+ x}\psi_{y- x}(b)]_{\lambda+ \mu+ y}c]\\
	&- [[\phi_x(a)_{\lambda+ x}b]_{\lambda+ \mu+ x}\psi_{y- x}(c)])\\
	= &0.
	\end{align*}
	Similarly, we can obtain that
	\begin{align*}
	[\phi_x\psi]_y([[a_{\lambda}b]_{\lambda+ \mu}c])= 0.
	\end{align*}
	Hence, we deduce that $[\phi_x\psi]_y \in ZTDer(\mathcal{R})[x]$.
	
	This completes the proof.
\end{proof}

\begin{remark}
	\em{
		It is well known that $gc(\mathcal{R})$ is a simple Lie conformal algebra. Thus, if $GCTDer(\mathcal{R})= gc(\mathcal{R})$, by Proposition \ref{pro3.5}, $ZTDer(\mathcal{R})$ is equal to either 0 or $gc(\mathcal{R})$. }

\end{remark}

\begin{proposition}\label{pro3.7}
	If $Z(\mathcal{R})= 0$, then $TC(\mathcal{R})$ and $TQC(\mathcal{R})$ are commutative Lie conformal algebras. In particular, if $Z(\mathcal{R})= 0$, $TC(\mathcal{R})$ centralizes $TQC(\mathcal{R})$.
\end{proposition}
\begin{proof}
	Clearly, we only need to prove the first case, the other case can be proved similarly. For any $a,b,c \in \mathcal{R}$, $\phi_x,\psi_x\in TC(\mathcal{R})$, we can deduce that
	\begin{align*}
	[[\phi_x\psi_{y- x}(a)_{\lambda+ y}b]_{\lambda+ \mu+ y}c]&= \phi_x([[\psi_{y- x}(a)_{\lambda+ y- x}b]_{\lambda+ \mu+ y- x}c])\\
	&= \phi_x([[a_{\lambda}\psi_{y- x}(b)]_{\lambda+ \mu+ y- x}c])\\
	&= [[\phi_x(a)_{\lambda+ x}\psi_{y- x}(b)]_{\lambda+ \mu+ y}c]\\
	&= [[\psi_{y- x}\phi_x(a)_{\lambda+ y}b]_{\lambda+ \mu+ y}c].
	\end{align*}
	Hence, we have
	\begin{align*}
	[[[\phi_x\psi]_y(a)_{\lambda+ y}b]_{\lambda+ \mu+ y}c]= 0.
	\end{align*}
	Due to the arbitrary of $b, c$ and $Z(\mathcal{R})= 0$, we can deduce that
	\begin{align*}
	[\phi_x\psi]_y(a)= 0.
	\end{align*}
	Thus, we have $[\phi_x\psi]_y= 0$.
	
	This completes the proof.
\end{proof}

\begin{proposition}\label{pro3.8}
	Let $\mathcal{R}$ be a Lie conformal algebra. Then we have the following results:
	\begin{enumerate}
		\item $[CTDer(\mathcal{R})_\lambda TC(\mathcal{R})]_\mu \subseteq TC(\mathcal{R})[\lambda]$.
		\item $[GCTDer(\mathcal{R})_\lambda TC(\mathcal{R})]_\mu \subseteq TC(\mathcal{R})[\lambda]$.
		\item $[CTDer(\mathcal{R})_\lambda TQC(\mathcal{R})]_\mu \subseteq TQC(\mathcal{R})[\lambda]$.
	\end{enumerate}
\end{proposition}
\begin{proof}
	
	Obviously, we only need to prove the first case, other cases can be proved similarly.
	Since $\phi_x\in CTDer(\mathcal{R})$, $\psi_x \in TC(\mathcal{R})$, for any $a, b, c \in \mathcal{R}$, we can obtain that
	\begin{align*}
	&[[\phi_x\psi_{y- x}(a)_{\lambda+ y}b]_{\lambda+ \mu+ y}c]\\
	= &\phi_x([[\psi_{y-x}(a)_{\lambda+ y- x}b]_{\lambda+ \mu+ y- x}c])-  [[\psi_{y- x}(a)_{\lambda+ y- x}\phi_x(b)]_{\lambda+ \mu+ y}c]\\
	&- [[\psi_{y- x}(a)_{\lambda+ y- x}b]_{\lambda+ \mu+ y- x}\phi_x(c)]\\
	= &\phi_x\psi_{y- x}([[a_{\lambda}b]_{\lambda+ \mu}c])- \psi_{y- x}([[a_{\lambda}\phi_x(b)]_{\lambda+ \mu+ x}c])- \psi_{y- x}([[a_{\lambda}b]_{\lambda+ \mu}\phi_x(c)]).
	\end{align*}
	Similarly, we have
	\begin{align*}
	&[[\psi_{y- x}\phi_x(a)_{\lambda+ y}b]_{\lambda+ \mu+ y}c]\\
	= &\psi_{y- x}([[\phi_x(a)_{\lambda+ x}b]_{\lambda+ \mu+ x}c])\\
	= &\psi_{y- x}\phi_x([[a_{\lambda}b]_{\lambda+ \mu}c])- \psi_{y- x}([[a_{\lambda}\phi_x(b)]_{\lambda+ \mu+ x}c])- \psi_{y- x}([[a_{\lambda}b]_{\lambda+ \mu}\phi_x(c)]).
	\end{align*}
	Hence, we can obtain that
	\begin{align*}
	[[[\phi_x\psi]_y(a)_{\lambda+ y}b]_{\lambda+ \mu+ y}c]= [\phi_x\psi]_y([[a_{\lambda}b]_{\lambda+ \mu}c]).
	\end{align*}
	Thus, we have $[\phi_x\psi]_y \in TC(\mathcal{R})$.
	
	This completes the proof.
\end{proof}

\section{Classification of conformal triple derivations on simple Lie conformal algebras}
In this section, we classify (generalized) conformal triple derivations of all finite simple Lie conformal algebras. For convenience, we always assume that $\mathcal{R}$ is a finite simple Lie conformal algebra in this section unless otherwise specified.

It is easy to see that $CDer(\mathcal{R})$ and $CInn(\mathcal{R})$ are Lie conformal subalgebras of $CTDer(\mathcal{R})$. Furthermore, we have the following lemma.
\begin{lemma}\label{lm4.1}
	$CInn(\mathcal{R})$ is an ideal of Lie conformal algebra $CTDer(\mathcal{R})$.
\end{lemma}
\begin{proof}
	Let $\phi_x \in CTDer(\mathcal{R})$, $a \in \mathcal{R}$. Since $\mathcal{R}$ is a simple Lie conformal algebra, then there exists some finite index set $I, J \subseteq \mathbb{Z^+}$ and $a_{ij}^1, a_{ij}^2 \in \mathcal{R}$ such that $a= \sum_{i\in I, j\in J}\frac{d^i}{d\lambda^i}[{a_{ij}^1}_\lambda {a_{ij}^2}]|_{\lambda=0}$.
	
	For any $b\in \mathcal{R}$, we have
	\begin{align*}
	[\phi_x ad a]_y(b)&= \phi_x ada_{y- x}(b)- ada_{y- x}\phi_x(b)\\
	= &\phi_x ([a_{y- x}b])- [a_{y- x}\phi_x(b)]\\
	= &\phi_x([(\sum_{i\in I, j\in J}\frac{d^i}{d\lambda^i}[{a_{ij}^1}_\lambda {a_{ij}^2}]|_{\lambda=0})_{y- x}b])- [(\sum_{i\in I, j\in J}\frac{d^i}{d\lambda^i}[{a_{ij}^1}_\lambda {a_{ij}^2}]|_{\lambda=0})_{y- x}\phi_x(b)]\\
	= &\sum_{i\in I, j\in J}\frac{d^i}{d\lambda^i}([[\phi_x(a_{ij}^1)_{\lambda+ x} {a_{ij}^2}]_{y}b]+ [[{a_{ij}^1}_\lambda \phi_x({a_{ij}^2})]_{y}b]+ [[{a_{ij}^1}_\lambda {a_{ij}^2}]_{y- x}\phi_x(b)])|_{\lambda=0}\\
	&- [(\sum_{i\in I, j\in J}{\frac{d^i}{d\lambda^i}[{a_{ij}^1}_\lambda {a_{ij}^2}]|_{\lambda=0}})_{y- x}\phi_x(b)]\\
	= &\sum_{i\in I, j\in J}\frac{d^i}{d\lambda^i}([[\phi_x(a_{ij}^1)_{\lambda+ x} {a_{ij}^2}]_{y}b]+ [[{a_{ij}^1}_\lambda \phi_x({a_{ij}^2})]_{y}b])|_{\lambda=0}\\
	= &ad(\sum_{i\in I, j\in J}\frac{d^i}{d\lambda^i}([\phi_x(a_{ij}^1)_{\lambda+ x} {a_{ij}^2}]+ [ {a_{ij}^1}_\lambda \phi_x({a_{ij}^2})])|_{\lambda=0})_{y}(b).
	\end{align*}
	By the arbitrariness of $b$, $[\phi_x ad a]_y$ is an inner conformal derivation. Thus, $CInn(\mathcal{R})$ is an ideal of Lie conformal algebra $CTDer(\mathcal{R})$.
\end{proof}

In fact, there is another connection between $CDer(\mathcal{R})$ and $CTDer(\mathcal{R})$, which is stronger than the relationship between them as shown in the above lemma. Thus, we have the following lemma.

\begin{lemma}\label{lm4.2}
	There exists a $\mathbb{C}$-linear map $\delta: CTDer(\mathcal{R})\to CDer(\mathcal{R})$, $\phi_x \mapsto \delta_{\phi_x}$ such that for all $a \in \mathcal{R}$, $\phi_x \in CTDer(\mathcal{R})$, one has $[\phi_x ad a]_y= ad {\delta_{\phi_x}}(a)_y$.
\end{lemma}
\begin{proof}
	By the proof of Lemma \ref{lm4.1}, we  define a conformal linear endomorphism $\delta_{\phi_x}$ on $\mathcal{R}$, such that for any $a= \sum_{i\in I, j\in J}\frac{d^i}{d\lambda^i}[{a_{ij}^1}_\lambda {a_{ij}^2}]|_{\lambda=0} \in \mathcal{R}$,
	\begin{align*}
	\delta_{\phi_x}(a)= \sum_{i\in I, j\in J}\frac{d^i}{d\lambda^i}([\phi_x(a_{ij}^1)_{\lambda+ x} {a_{ij}^2}]+ [ {a_{ij}^1}_\lambda \phi_x({a_{ij}^2})])|_{\lambda=0}.
	\end{align*}
	At first, $\delta_{\phi_x}(a)$ does not depend on the choice of expression of $a$. For proving it, take
	\begin{align*}
	A&= \sum_{i\in I, j\in J}\frac{d^i}{d\lambda^i}([\phi_x(a_{ij}^1)_{\lambda+ x} {a_{ij}^2}]+ [ {a_{ij}^1}_\lambda \phi_x({a_{ij}^2})])|_{\lambda=0},\\
	B&= \sum_{s\in I^{'}, t\in J^{'}}\frac{d^s}{d\lambda^s}([\phi_x(b_{st}^1)_{\lambda+ x} {b_{st}^2}]+ [ {b_{st}^1}_\lambda \phi_x({b_{st}^2})])|_{\lambda=0},
	\end{align*}
	where $a$ can also be expressed in the form $a=  \sum_{s\in I^{'}, t\in J^{'}}\frac{d^s}{d\lambda^s}[{b_{st}^1}_\lambda {b_{st}^2}]|_{\lambda=0}$. Since $\phi_x \in CTDer(\mathcal{R})$, for any $c \in \mathcal{R}$, we can deduce that
	\begin{align*}
	[A_{\mu+ x}c]= \phi_x([a_{\mu}c])- [a_{\mu}\phi_x(c)]= [B_{\mu+ x}c].
	\end{align*}
	Hence, $[(A- B)_{\mu+ x}c]= 0$, for any $c \in \mathcal{R}$, i.e. $A- B\in Z(\mathcal{R})$. Obviously, we have $Z(\mathcal{R})= 0$, because $\mathcal{R}$ is a simple Lie conformal algebra. Then, we obtain that $A= B$. Therefore, $\delta_{\phi_x}$ is well-defined. Furthermore, it follows immediately from the proof of Lemma \ref{lm4.1}  that $[\phi_x ad a]_y= ad {\delta_{\phi_x}}(a)_y$.
	
	Next, we need to prove that $\delta_{\phi_x}$ is a conformal linear map, i.e. $\delta_{\phi_x}(\partial a)= (\partial+ x)\delta_{\phi_x}(a)$. For any $a= \sum_{i\in I, j\in J}\frac{d^i}{d\lambda^i}[{a_{ij}^1}_\lambda {a_{ij}^2}]|_{\lambda=0} \in \mathcal{R}$, we have $\partial a= \sum_{i\in I, j\in J}\frac{d^i}{d\lambda^i}([{\partial a_{ij}^1}_\lambda {a_{ij}^2}] + [{a_{ij}^1}_\lambda \partial {a_{ij}^2}])|_{\lambda=0}$. Thus, we can obtain that
	\begin{align*}
	\delta_{\phi_x}(\partial a)
	=& \sum_{i\in I, j\in J}\frac{d^i}{d\lambda^i}([\phi_x(\partial a_{ij}^1)_{\lambda+ x} {a_{ij}^2}]+ [{\partial a_{ij}^1}_\lambda \phi_x({a_{ij}^2})]+ [\phi_x(a_{ij}^1)_{\lambda+ x} {\partial  a_{ij}^2}]+ [{a_{ij}^1}_\lambda \phi_x({\partial  a_{ij}^2})])|_{\lambda=0}\\
	=& \sum_{i\in I, j\in J}\frac{d^i}{d\lambda^i}(- \lambda([\phi_x(a_{ij}^1)_{\lambda+ x} {a_{ij}^2}]+ [{a_{ij}^1}_\lambda \phi_x({a_{ij}^2})])\\
	&+ (\partial +\lambda+ x)([\phi_x(a_{ij}^1)_{\lambda+ x} { a_{ij}^2}]+ [{a_{ij}^1}_\lambda \phi_x({a_{ij}^2})]))|_{\lambda=0}\\
	=& (\partial+ x)\sum_{i\in I, j\in J}\frac{d^i}{d\lambda^i}([\phi_x(a_{ij}^1)_{\lambda+ x} { a_{ij}^2}]+ [{a_{ij}^1}_\lambda \phi_x({a_{ij}^2})])|_{\lambda=0}\\
	=& (\partial+ x)\delta_{\phi_x}(a).
	\end{align*}
	
	Finally, we  prove that $\delta_{\phi_x}$ is a conformal derivation. By the above discussion, for any  $\phi_x \in CTDer(\mathcal{R})$, $a, b \in \mathcal{R}$, we have $[\phi_x ad[a_\lambda b]]= ad \delta_{\phi_x} ([a_\lambda b])$. On the other hand, we have
	\begin{align*}
	[\phi_x ad[a_\lambda b]]
	&= [\phi_x [ad a_\lambda ad b]]\\
	&= [[\phi_x ad a]_{\lambda+x} ad b]+ [ad a_\lambda [\phi_x ad b]]\\
	&= [ad \delta_{\phi_x}(a)_{\lambda+x} ad b]+ [ad a_{\lambda} ad \delta_{\phi_x}(b)]\\
	&= ad ([\delta_{\phi_x}(a)_{\lambda+x} b]+ [a_{\lambda} \delta_{\phi_x}(b)]).
	\end{align*}
	Therefore, $ad \delta_{\phi_x} ([a_\lambda b])= ad ([\delta_{\phi_x}(a)_{\lambda+ x} b]+ [a_{\lambda} \delta_{\phi_x}(b)])$. Since $Z(\mathcal{R})= 0$,  we have $\delta_{\phi_x} ([a_\lambda b])= [\delta_{\phi_x}(a)_{\lambda+ x} b]+ [a_{\lambda} \delta_{\phi_x}(b)]$. By the arbitrariness of $a, b$, we can deduce that $\delta_{\phi_x}  \in  CDer(\mathcal{R})$.
	
	This completes the proof.
\end{proof}

\begin{lemma}\label{lm4.3}
	The centralizer of $CInn(\mathcal{R})$ in $CTDer(\mathcal{R})$ is trivial, i.e., $C_{CTDer(\mathcal{R})}(CInn(\mathcal{R}))= 0$. In particular, the center of $CTDer(\mathcal{R})$ is zero.
\end{lemma}
\begin{proof}
	Suppose that $\phi_x \in C_{CTDer(\mathcal{R})}(CInn(\mathcal{R}))$. Then we have $[\phi_x ad a]= 0$ for any $a \in \mathcal{R}$. Thus, for any $b \in \mathcal{R}$, we can obtain that $\phi_x ([a_{\lambda}b])- [a_\lambda \phi_x (b)]= [\phi_x ad a]_{\lambda+ x}(b)= 0$. Moreover, we have $\phi_x ([a_{\lambda}b])= [a_\lambda \phi_x (b)]= [\phi_x (a)_{\lambda+ x} b]$.
	
	On the one hand, we have
	\begin{align*}
	\phi_x([[a_{\lambda}b]_{\lambda+ \mu}c])= [[\phi_x(a)_{\lambda+ x}b]_{\lambda+ \mu+ x}c]= [[a_{\lambda}\phi_x(b)]_{\lambda+ \mu+ x}c]= [[a_{\lambda}b]_{\lambda+ \mu}\phi_x(c)],
	\end{align*}
	for any $a, b, c \in \mathcal{R}$. On the other hand, we have
	\begin{align*}
	\phi_x([[a_{\lambda}b]_{\lambda+ \mu}c])= [[\phi_x(a)_{\lambda+ x}b]_{\lambda+ \mu+ x}c]+ [[a_{\lambda}\phi_x(b)]_{\lambda+ \mu+ x}c]+ [[a_{\lambda}b]_{\lambda+ \mu}\phi_x(c)],
	\end{align*}
	because $\phi_x \in CTDer(\mathcal{R})$. Hence, we can obtain that
	\begin{align*}
	\phi_x([[a_{\lambda}b]_{\lambda+\mu}c])=3\phi_x([[a_{\lambda}b]_{\lambda+ \mu}c]).
	\end{align*}
	This means $\phi_x([[a_{\lambda}b]_{\lambda+ \mu}c])= 0$ for any $a, b, c \in \mathcal{R}$. Since $\mathcal{R}$ is a simple Lie conformal algebra, every element of $\mathcal{R}$ can be expressed as the form $\sum_{i\in I, j\in J, k \in K}\frac{d^i}{d\lambda^i}\frac{d^k}{d\mu^k}[[{a_{ijk}^1}_\lambda {a_{ijk}^2}]_\mu {a_{ijk}^3} ]|_{\lambda= \mu= 0}$.
	Then we deduce that $\phi_x= 0$.
\end{proof}

Now, let us introduce the main result in this section by the following theorem, which will give the classification of the conformal triple derivations of a finite simple Lie conformal algebra.
\begin{theorem}\label{tm4.4}
	$CTDer(\mathcal{R})= CDer(\mathcal{R})$.
\end{theorem}
\begin{proof}
	Assume that $\phi_x \in CTDer(\mathcal{R})$, $a \in \mathcal{R}$. By Lemma \ref{lm4.2}, we have $[\phi_x ad a]_y= ad {\delta_{\phi_x}}(a)_y$ and $\delta_{\phi_x} \in CDer(\mathcal{R})$. Moreover, for any $b \in \mathcal{R}$, we have $ad {\delta_{\phi_x}}(a)_y (b)= [\delta_{\phi_x}(a)_y b]= \delta_{\phi_x}([a_{y- x}b])- [a_{y- x}\delta_{\phi_x}(b)]= [\delta_{\phi_x} ada]_y(b)$. By the arbitrariness of $b$, we can deduce that $ad {\delta_{\phi_x}}(a)_y= [\delta_{\phi_x} ada]_y$. Hence, we have $[\phi_x ad a]_y= [\delta_{\phi_x} ada]_y$, i.e., $\phi_x- \delta_{\phi_x} \in C_{CTDer(\mathcal{R})}(CInn(\mathcal{R}))$. By Lemma \ref{lm4.3}, $\phi_x- \delta_{\phi_x}= 0$, i.e., $\phi_x= \delta_{\phi_x}$. Thus, the theorem follows from Lemma \ref{lm4.2}.\end{proof}

In \cite{DK}, the classification of all finite simple Lie conformal algebras was given as follows by D'Andrea and Kac.
\begin{proposition}
	A finite simple Lie conformal algebra is isomorphic either to $Vir$ or to the current conformal algebra $Cur\mathcal{G}$ associated to a finite dimensional simple Lie algebra $\mathcal{G}$.
\end{proposition}
It was also shown in \cite{DK} that conformal derivations on $Vir$ and all simple current Lie conformal algebras $Cur\mathcal{G}$ are as follows.\begin{proposition}
	(1) Every conformal derivation of  $Vir$ is inner. \\
	(2) For a finite dimensional simple Lie algebra $\mathcal{G}$, every conformal derivation of $Cur\mathcal{G}$ is of the form $p(\partial)d^L+d,$ where $d$ is an inner conformal derivation and $d^L$ is as in Example \ref{ex2.6}.
\end{proposition}

Together with Theorem \ref{tm4.4} and the above two propositions, we immediately obtain the following results .

\begin{corollary}\label{cor4.7}
	(1) $CTDer(Vir)= CInn(Vir)$.\\
	(2) $CTDer(Cur\mathcal{G})= \{ \phi_x \in gc(Cur\mathcal{G}) | \phi_x(a)= f(x)(\partial+ x)a+ d_x(a), \  where \ a\in \mathcal{G}, \ f(x)\in \mathbb{C}[x] \ and \ d_x \in CInn(Cur\mathcal{G}) \}$, where $\mathcal{G}$ is a finite dimensional simple Lie algebra.
\end{corollary}

Inspired by Lemma \ref{lm3.1} and Corollary \ref{cor4.7}, we give the classification of generalized conformal triple derivations on a finite simple Lie conformal algebra.

\begin{theorem}\label{tm4.8}
	(1) $GCTDer(Vir)= CTDer(Vir)= CInn(Vir)$.\\
	(2) $GCTDer(Cur\mathcal{G})= \{\phi_x \in gc(Cur\mathcal{G}) | \phi_x(a)= (f(x)\partial+ g(x))a+ d_x(a), \ where \ a\in \mathcal{G}, \ f(x), g(x)\in \mathbb{C}[x] \ and \ d_x \in CInn(Cur\mathcal{G}) \}$, where $\mathcal{G}$ is a finite dimensional simple Lie algebra.
\end{theorem}
\begin{proof}
	(1) It can be immediately obtained from Remark \ref{rm3.5}, Proposition \ref{pro3.33} and Corollary \ref{cor4.7}.
	
	(2) Inspired by Remark \ref{rm3.5}, we only need to figure out $TC(Cur\mathcal{G})$. For any $\phi_x \in TC(Cur\mathcal{G})$, since $Cur\mathcal{G}$ is a free $\mathbb{C}[\partial]$-module of finite rank, we assume that $\phi_x(a)=\sum_{i=0}^n\partial^i f_x^i(a)$ for every $a \in \mathcal{G}$, where $f_x^i$ are $\mathbb{C}$-linear maps from $\mathcal{G}$ to $\mathcal{G}[x]$. By (\ref{eq3.3}), for any $a, b, c \in \mathcal{G}$, we obtain that
	\begin{align*}
	\sum_{i= 0}^n\partial^i f_x^i([[a, b], c])= \sum_{i= 0}^n(- \lambda- x)^i [[f_x^i(a), b], c].
	\end{align*}
	Obviously, we can deduce that $n= 0$ by comparing the coefficient of $\partial$ in the above equation. Thus, we set $\phi_x(a)= f_x(a)$, for every $a \in \mathcal{G}$, where $f_x$ is a $\mathbb{C}$-linear map from $\mathcal{G}$ to $\mathcal{G}[x]$. According to the conclusion (iii) of (1) in Remark \ref{rm3.3}, we also have
	\begin{align}\label{4.81}
	f_x([[a, b], c])= [[f_x(a), b], c]= [[a, f_x(b)], c]= [[a, b], f_x(c)].
	\end{align}
	Since $\mathcal{G}$ is a finite
dimensional simple Lie algebra, then $Z(\mathcal{G})= 0$. By (\ref{4.81}), we have $[([f_x(a), b]- [a, f_x(b)]), c]= 0$, for any $c \in \mathcal{G}$. Thus, we have
	\begin{align}\label{4.82}
	[f_x(a), b]= [a, f_x(b)]
	\end{align}
	for any $a, b \in \mathcal{G}$.
	Suppose that $f_x(a)= \sum_{i= 0}^m x^i f_i(a)$ for every $a \in \mathcal{G}$, where $f_i$ are $\mathbb{C}$-linear maps from $\mathcal{G}$ to $\mathcal{G}$. Plugging $f_x(a)$ into (\ref{4.82}), we get
	\begin{align}\label{4.83}
	[f_i(a), b]= [a, f_i(b)]
	\end{align}
	for all $i \in \{0, 1, ..., m\}$. By Lemma 6.3 in \cite{DK}, we can deduce that $f_i= k_iId_{\mathcal{G}}$, where $k_i \in \mathbb{C}$. Hence, we have $\phi_x(a)= g(x)a$ for any $g(x) \in \mathbb{C}[x]$.
	
	Thus, (2) can be directly obtained by Theorem \ref{tm4.8} and the above discussion.
	
	This completes the proof.
\end{proof}

\section{Classification of triple homomorphisms on simple Lie conformal algebras}

In this section, we first introduce the definition of triple homomorphism of a Lie conformal algebra. Then the triple homomorphisms of all finite simple Lie conformal algebras are also characterized.

\begin{definition}\label{def5.1}
\begin{em}
Let $\mathcal{A}, \mathcal{B}$ be two Lie conformal algebras. A $\mathbb{C}[\partial]$-module homomorphism $f: \mathcal{A} \to \mathcal{B}$ is called:
\begin{enumerate}
			\item a \emph{homomorphism} if it satisfies $f([a_\lambda b])= [f(a)_\lambda f(b)]$ for any $a, b \in \mathcal{A}$.
			\item an \emph{anti-homomorphism} if it satisfies $f([a_\lambda b])= -[f(a)_\lambda f(b)]$ for any $a,b \in \mathcal{A}$.
			\item a \emph{triple homomorphism} if it satisfies $f([a_\lambda [b_\mu c]])= [f(a)_\lambda [f(b)_\mu f(c)]]$ for any $a, b, c \in \mathcal{A}$.\end{enumerate}
\end{em}
\end{definition}

\begin{remark}

	\em{It is not difficult to see that both homomorphisms and anti-homomorphisms of Lie conformal algebra are triple homomorphisms, but the converse does not always hold. Moreover, the sum of a homomorphism and an anti-homomorphism is also a triple homomorphism. Actually, there exists some other connection between them which is not just inclusion relationship.}
\end{remark}
\begin{remark}\label{rm5.2}
	\em{Obviously, due to skew-symmetry of Lie conformal algebra, the following two equalities are equivalent:
		\begin{enumerate}
			\item $f([a_\lambda [b_\mu c]])= [f(a)_\lambda [f(b)_\mu f(c)]]$.
			\item $f([[a_\lambda b]_{\lambda+ \mu} c])= [[f(a)_\lambda f(b)]_{\lambda+ \mu} f(c)]$.
		\end{enumerate}}
\end{remark}

Let $\mathcal{A}$ be a Lie conformal algebra. For a subset $S$ of $\mathcal{A}$, the \emph{enveloping Lie conformal algebra} of $S$ is by definition the Lie conformal subalgebra of $\mathcal{A}$ generated by $S$. A Lie conformal algebra is called \emph{indecomposable} if it cannot be written as a direct sum of two nontrivial ideals.

\begin{definition}\label{def5.2}
	\em{
		Suppose that $\mathcal{A} ~and~ \mathcal{B}$ are two Lie conformal algebras. A $\mathbb{C}[\partial]$-module homomorphism $f: \mathcal{A} \to \mathcal{B}$ is called a \emph{direct sum} of $f_1$ and $f_2$, if $f= f_1+ f_2$ and there exists ideals $I_1$, $I_2$ of the enveloping Lie conformal algebra of $f(\mathcal{A})$ such that $I_1\cap I_2= 0$ and $f_1(\mathcal{A}) \subseteq I_1$, $f_2(\mathcal{A}) \subseteq I_2$.}
		
\end{definition}

Hereafter, we always assume that $\mathcal{A} ~and~ \mathcal{B}$ are Lie conformal algebras and $f$ is a triple homomorphism from $\mathcal{A}$ to $\mathcal{B}$. And denote by $E$ the enveloping Lie conformal algebra of $f(\mathcal{A})$. Moreover, we suppose that $\mathcal{A}$ is a finite simple Lie conformal algebra, and $E$ is centerless and can be decomposed into a direct sum of indecomposable ideals.

\begin{lemma}\label{lm5.4}
	Let $\mathcal{A}$ be a finite simple Lie conformal algebra and $\mathcal{B}$ a Lie conformal algebra, there exists a  homomorphism $\delta_f: \mathcal{A} \to \mathcal{B}$ such that for all $a \in \mathcal{A}$ one has $fada_x= ad \delta_f(a)_xf$.
\end{lemma}
\begin{proof}
	Similar to the discussion of the proof of Lemma \ref{lm4.1}, since $\mathcal{A}$ is a finite simple Lie conformal algebra, there exists some finite index set $I, J \subseteq \mathbb{Z^+}$ and $a_{ij}^1, a_{ij}^2 \in \mathcal{A}$ such that $a= \sum_{i\in I, j\in J}\frac{d^i}{d\lambda^i}[{a_{ij}^1}_\lambda {a_{ij}^2}]|_{\lambda=0}$. Define a $\mathbb{C}$-linear map $\delta_f: \mathcal{A} \to \mathcal{B}$, such that for any $a \in \mathcal{A}$, $\delta_f(a)= \sum_{i\in I, j\in J}\frac{d^i}{d\lambda^i}[f(a_{ij}^1)_{\lambda} f({a_{ij}^2})]|_{\lambda=0}$.

	At first, it is sufficient to prove that  $\delta_f(a)$ is independent of the expression of $a$. For proving it, we take
	\begin{align*}
	A= \sum_{i\in I, j\in J}\frac{d^i}{d\lambda^i}[f(a_{ij}^1)_{\lambda} f({a_{ij}^2})]|_{\lambda=0}, \quad B=\sum_{s\in I^{'}, t\in J^{'}}\frac{d^s}{d\lambda^s}[f(b_{st}^1)_{\lambda} f({b_{st}^2})]|_{\lambda=0},
	\end{align*}
	where $a$ can also be expressed in the form $a= \sum_{s\in I^{'}, t\in J^{'}}\frac{d^s}{d\lambda^s}[{b_{st}^1}_{\lambda} {b_{st}^2}]|_{\lambda=0}$. Since $f$ is a triple homomorphism from $\mathcal{A}$ to $\mathcal{B}$, for any $c \in \mathcal{A}$, we can deduce that
	\begin{align*}
	[f(c)_\mu (A- B)]&= [f(c)_\mu \sum_{i\in I, j\in J}\frac{d^i}{d\lambda^i}[f(a_{ij}^1)_{\lambda} f({a_{ij}^2})]|_{\lambda=0}]]- [f(c)_\mu \sum_{s\in I^{'}, t\in J^{'}}\frac{d^s}{d\lambda^s}[f(b_{st}^1)_{\lambda} f({b_{st}^2})]|_{\lambda=0}]]\\
	&= \sum_{i\in I, j\in J}\frac{d^i}{d\lambda^i}([f(c)_\mu [f(a_{ij}^1)_{\lambda} f({a_{ij}^2})]])|_{\lambda=0}- \sum_{s\in I^{'}, t\in J^{'}}\frac{d^s}{d\lambda^s}([f(c)_\mu [f(b_{st}^1)_{\lambda} f({b_{st}^2})]])|_{\lambda=0}\\
	&= \sum_{i\in I, j\in J}\frac{d^i}{d\lambda^i}(f([c_\mu [{a_{ij}^1}_{\lambda} {a_{ij}^2}]]))|_{\lambda=0}- \sum_{s\in I^{'}, t\in J^{'}}\frac{d^s}{d\lambda^s}(f([c_\mu [{b_{st}^1}_{\lambda} {b_{st}^2}]]))|_{\lambda=0}\\
	&= f([c_\mu \sum_{i\in I, j\in J}\frac{d^i}{d\lambda^i}([{a_{ij}^1}_{\lambda} {a_{ij}^2}])|_{\lambda=0}])- f([c_\mu \sum_{s\in I^{'}, t\in J^{'}}\frac{d^s}{d\lambda^s}([{b_{st}^1}_{\lambda} {b_{st}^2}])|_{\lambda=0}])\\
	&= f([c_\mu a])- f([c_\mu a])\\
	&= 0.
	\end{align*}
	
	Hence, $A- B \in Z(E)$. By the assumption, we have $Z(E)= 0$. Thus, we can obtain that $A= B$. Therefore, $\delta_f$ is well-defined.
	
	Next, we need to prove that $fada_x= ad \delta_f(a)_xf$ for each $a\in \mathcal{A}$. For any $a, b \in \mathcal{A}$, where $a= \sum_{i\in I, j\in J}\frac{d^i}{d\lambda^i}[{a_{ij}^1}_\lambda {a_{ij}^2}]|_{\lambda=0}$, we have
	\begin{align*}
	fada_x(b)&= f([a_x b])\\
	&= f([(\sum_{i\in I, j\in J}\frac{d^i}{d\lambda^i}[{a_{ij}^1}_\lambda {a_{ij}^2}]|_{\lambda=0})_x b])\\
	&= \sum_{i\in I, j\in J}\frac{d^i}{d\lambda^i}[[f({a_{ij}^1})_\lambda f({a_{ij}^2})]_x f(b)]|_{\lambda=0}\\
	&= [(\sum_{i\in I, j\in J}\frac{d^i}{d\lambda^i}[f({a_{ij}^1})_\lambda f({a_{ij}^2})]|_{\lambda=0})_x f(b)]\\
	&= [\delta_f(a)_x f(b)]\\
	&= ad \delta_f(a)_x f(b).
	\end{align*}
	Therefore, by arbitrariness of $b\in \mathcal{A}$, for each $a\in \mathcal{A}$, we get $fada_x= ad \delta_f(a)_xf$.
	
	Finally, we prove that $\delta_{f}$ is a homomorphism, i.e. $\delta_{f}(\partial a)= \partial\delta_{f}(a)$ and $\delta_f([a_\lambda b])= [\delta_f(a)_\lambda \delta_f(b)]$ for any $a, b \in \mathcal{A}$. For any $a= \sum_{i\in I, j\in J}\frac{d^i}{d\lambda^i}[{a_{ij}^1}_\lambda {a_{ij}^2}]|_{\lambda=0} \in \mathcal{A}$, we have $\partial a= \sum_{i\in I, j\in J}\frac{d^i}{d\lambda^i}([{\partial a_{ij}^1}_\lambda {a_{ij}^2}] + [{a_{ij}^1}_\lambda \partial {a_{ij}^2}])|_{\lambda=0}$. Thus, we obtain that
	\begin{align*}
	\delta_{f}(\partial a)&= \sum_{i\in I, j\in J}\frac{d^i}{d\lambda^i}([f(\partial a_{ij}^1)_{\lambda} f({a_{ij}^2})]+ [f({a_{ij}^1})_\lambda f({\partial a_{ij}^2})])|_{\lambda=0}\\
	&= \sum_{i\in I, j\in J}\frac{d^i}{d\lambda^i}(-\lambda [f(a_{ij}^1)_{\lambda} f({a_{ij}^2})]+ (\partial + \lambda)[f({a_{ij}^1})_\lambda f(a_{ij}^2)])|_{\lambda=0}\\
	&= \partial\sum_{i\in I, j\in J}\frac{d^i}{d\lambda^i}([f(a_{ij}^1)_{\lambda} f({a_{ij}^2})])|_{\lambda=0}\\
	&= \partial\delta_{f}(a).
	\end{align*}
	Thus, $\delta_{f}$ is a $\mathbb{C}[\partial]$-module homomorphism. For any $a,b,c \in \mathcal{A}$, we  obtain that
	\begin{align*}
	[&(\delta_{f}([a_\lambda b])- [\delta_{f}(a)_\lambda \delta_f(b)])_{\lambda+ \mu} f(c)]\\
	&= [\delta_{f}([a_\lambda b])_{\lambda+ \mu} f(c)]- [\delta_{f}(a)_{\lambda} [\delta_f(b)_{\mu} f(c)]]+ [\delta_f(b)_{\mu} [\delta_{f}(a)_\lambda f(c)]]\\
	&= [[f(a)_\lambda f(b)]_{\lambda+ \mu} f(c)]- [\delta_{f}(a)_{\lambda} (ad \delta_f(b)_\mu f(c))]+ [\delta_f(b)_{\mu}(ad \delta_f(a)_\lambda f(c))]\\
	&= [[f(a)_\lambda f(b)]_{\lambda+ \mu} f(c)]- [\delta_{f}(a)_{\lambda} (fadb_\mu c)]+ [\delta_f(b)_{\mu}(fada_\lambda c)]\\
	&= f([[a_\lambda b]_{\lambda+ \mu} c])- [\delta_{f}(a)_{\lambda} f([b_\mu c])]+ [\delta_f(b)_{\mu}f([a_\lambda c])]\\
	&= f([[a_\lambda b]_{\lambda+ \mu} c])- ad\delta_{f}(a)_{\lambda} f([b_\mu c])+ ad\delta_f(b)_{\mu}f([a_\lambda c])\\
	&= f([[a_\lambda b]_{\lambda+ \mu} c])- fada_{\lambda}([b_\mu c])+ fadb_{\mu}([a_\lambda c])\\
	&= f([[a_\lambda b]_{\lambda+ \mu} c])- f([a_{\lambda}[b_\mu c]])+ f([b_{\mu}[a_\lambda c]])\\
	&= f([[a_\lambda b]_{\lambda+ \mu} c]- [a_{\lambda}[b_\mu c]]+ [b_{\mu}[a_\lambda c]])\\
	&= 0.
	\end{align*}
	Due to the arbitrariness of $c$, we have $\delta_{f}([a_\lambda b])- [\delta_{f}(a)_\lambda \delta_f(b)] \in Z(E)[\lambda]$. Thus, $\delta_{f}([a_\lambda b])= [\delta_{f}(a)_\lambda \delta_f(b)]$ since $E$ is centerless. Hence, $\delta_{f}$ is a  homomorphism by the arbitrariness of $a, b\in \mathcal{A}$.
	
	This completes the proof.
\end{proof}

There are some good properties on $E$ as shown in the following lemma.
\begin{lemma}\label{lm5.6}
	Denote $E^+= Im(f+ \delta_{f}), E^-= Im(f- \delta_{f})$. Then we obtain the following results:
	\begin{enumerate}
		\item $E^+$ and $E^-$ are both ideals of $E$.
		\item $[{E^+}_\lambda E^-]= 0$.
		\item $E^+ \cap E^-= 0$.
	\end{enumerate}
	In fact, $E$ can be decomposed into a direct sum of ideals $E^+$ and $E^-$.
\end{lemma}
\begin{proof}
	(1) Obviously, we have $E^+$, $E^- \subseteq E$. For any $a, b \in \mathcal{A}$, we have
	\begin{align*}
	[(f(a)+ \delta_f(a))_\lambda f(b)]
	&= [f(a)_\lambda f(b)]+ ad \delta_f(a)_\lambda f(b)\\
	&= \delta_f([a_\lambda b])+ fada_\lambda (b)\\
	&= \delta_f([a_\lambda b])+ f([a_\lambda b])\\
	&= (f+ \delta_{f})([a_\lambda b]).
	\end{align*}
	Thus, $E^+$ is an ideal of $E$. Similarly, $E^-$ is also an ideal of $E$.
	
	(2) For any $a, b, c \in \mathcal{A}$, we have 
	\begin{align*}
	[[(f(a)&+ \delta_f(a))_\lambda (f(b)- \delta_f(b))]_{\lambda+ \mu}f(c)]\\
	= &[[f(a)_\lambda f(b)]_{\lambda+ \mu}f(c)]- [[f(a)_\lambda \delta_f(b)]_{\lambda+ \mu}f(c)]\\
	&+ [[\delta_f(a)_\lambda f(b)]_{\lambda+ \mu}f(c)]- [[\delta_f(a)_\lambda \delta_f(b)]_{\lambda+ \mu}f(c)]\\
	= &f([[a_\lambda b]_{\lambda+ \mu}c])+ [[\delta_f(b)_{- \partial- \lambda} f(a)]_{\lambda+ \mu}f(c)]+ [[\delta_f(a)_\lambda f(b)]_{\lambda+ \mu}f(c)]\\
	&- [\delta_f(a)_\lambda [\delta_f(b)_{\mu}f(c)]]+ [\delta_f(b)_{\mu}[\delta_f(a)_\lambda f(c)]]\\
	= &f([[a_\lambda b]_{\lambda+ \mu}c])+ [(ad\delta_f(b)_{-\partial- \lambda} f(a))_{\lambda+ \mu}f(c)]+ [(ad \delta_f(a)_\lambda f(b))_{\lambda+ \mu}f(c)]\\
	&- [\delta_f(a)_\lambda (ad\delta_f(b)_{\mu}f(c))]+ [\delta_f(b)_{\mu}(ad\delta_f(a)_\lambda f(c))]\\
	= &f([[a_\lambda b]_{\lambda+ \mu}c])+ [(fadb_{-\partial- \lambda}(a))_{\lambda+ \mu}f(c)]+ [(fada_\lambda b)_{\lambda+ \mu}f(c)]\\
	&- [\delta_f(a)_\lambda (fadb_{\mu}c)]+ [\delta_f(b)_{\mu}(fada_\lambda c)]\\
	= &f([[a_\lambda b]_{\lambda+\mu}c])+ [f([b_{-\partial- \lambda}a])_{\lambda+\mu}f(c)]+ [f([a_\lambda b])_{\lambda+ \mu}f(c)]\\
	&-ad\delta_f(a)_\lambda f([b_{\mu}c])+ ad \delta_f(b)_{\mu}f([a_\lambda c])\\
	= &f([[a_\lambda b]_{\lambda+ \mu}c])- [f([a_\lambda b])_{\lambda+ \mu}f(c)]+ [f([a_\lambda b])_{\lambda+ \mu}f(c)]\\
	&- fada_\lambda ([b_{\mu}c])+ fadb_{\mu}([a_\lambda c])\\
	= &f([[a_\lambda b]_{\lambda+ \mu}c])- f([a_\lambda [b_{\mu}c]])+ f([b_{\mu}[a_\lambda c]])\\
	= &f([[a_\lambda b]_{\lambda+ \mu}c]- [a_\lambda [b_{\mu}c]]+ [b_{\mu}[a_\lambda c]])\\
	= &0.
	\end{align*}
	Therefore, $[(f(a)+ \delta_f(a))_\lambda (f(b)- \delta_f(b))] \in Z(E)$. Since $Z(E)= 0$, we obtain that $[(f(a)+ \delta_f(a))_\lambda (f(b)- \delta_f(b))]= 0$.
	
	(3) For any $a \in E^+\cap E^-$, we have $[a_\lambda E^+]= [a_\lambda E^-]= 0$ by (2). Thus, for any $b \in \mathcal{A}$, we have $[a_\lambda (f(b)+ \delta_f(b))]= [a_\lambda (f(b)- \delta_f(b))]= 0$. Then, $[a_\lambda f(b)]= 0$, i.e., $a \in Z(E)$. Since $Z(E)= 0$, we can deduce that $a= 0$.
	
	This completes the proof.
\end{proof}
We first consider the simplest situation of $E$ to investigate the triple homomorphism from $\mathcal{A}$ to $\mathcal{B}$.
\begin{lemma}\label{lm5.7}
	If $E$ is indecomposable, then $f$ is either a homomorphism or an anti-homomorphism from $\mathcal{A}$ to $\mathcal{B}$.
\end{lemma}
\begin{proof}
	For any $a \in \mathcal{A}$, set $a^+= \frac{1}{2}(f(a)+ \delta_f(a)), \ a^-= \frac{1}{2}(f(a)- \delta_f(a))$. Obviously, we have $a^+ \in E^+, a^-\in E^-$ and $f(a)= a^+ + a^-$. Therefore, $E \subseteq E^+ + E^-$. By Lemma \ref{lm5.6}, we have $E = E^+ \oplus  E^-$. Since $E$ is indecomposable, then either $E^+$ or $E^-$ must be trivial. If $E^+$ is trivial, then $f$ is an anti-homomorphism. If $E^-$ is trivial, then $f$ is a homomorphism.
	
	This completes the proof.
\end{proof}
Finally, we put forward the main result in this section by the following theorem.
\begin{theorem}\label{tm5.8}
	Assume that $\mathcal{A}$ is a finite simple Lie conformal algebra, and $E$ is centerless and can be decomposed into a direct sum of indecomposable ideals. Then $f$ is either a homomorphism, an anti-homomorphism or a direct sum of a homomorphism and an anti-homomorphism from $\mathcal{A}$ to $\mathcal{B}$.
\end{theorem}
\begin{proof}
	Obviously, it remains to prove Theorem \ref{tm5.8} in case $E$ is decomposable. By the assumption, $E$ can be written as the sum $E= E_1 \oplus E_2 \oplus ... \oplus E_n$, where each $E_i$ is an indecomposable ideal of $E$. Since $E$ is centerless, each $E_i$ is also centerless by Lemma \ref{lm5.6}.
	
	Let $p_i$ be the projection of $E$ into $E_i$. Then we have $f= \sum_{i= 1}^n p_if$ and $p_if$ is a triple homomorphism from $\mathcal{A}$ to $E_i$, and $E_i$ is the enveloping Lie conformal algebras of $p_if(\mathcal{A})$  for $i= 1, 2, ..., n$. Since each $E_i$ is indecomposable, by Lemma \ref{lm5.7}, we deduce that $p_i f$ is either a homomorphism or an anti-homomorphism from $\mathcal{A}$ to $E_i$. Set $I= \{i|p_if \ is\  a\  homomorphism.\}$ and $J= \{j| p_j f \ is\  an\ anti-homomorphism.\}$. Set $E_I=  \bigoplus_{i \in I} E_i$ and $E_J=\bigoplus_{j \in J} E_j$. Let $f_I=\sum_{i \in I}p_if$ and $f_J= \sum_{j \in J}p_j f$. It is not difficult to check that $E= E_I\oplus E_J$, $[{E_I}_\lambda E_J]= 0$, $f= f_I+ f_J$, where $f_I$ is a homomorphism from $\mathcal{A}$ to $E_I$ and $f_J$ is an anti-homomorphism from $\mathcal{A}$ to $E_J$.
	
	This completes the proof.
\end{proof}
\subsection*{Declaration:} This work is  accepted by the journal in 2021.
\subsection*{Acknowledgment:}This work was supported by the National Natural Science Foundation of China (No. 11871421, 12171129) and the Zhejiang Provincial Natural Science Foundation of China (No. LY20A010022) and the Scientific Research Foundation of Hangzhou Normal University (No. 2019QDL012) and the Fundamental Research Funds for the Central Universities (No. 22120210554).


%
\end{document}